\documentclass[10pt,a4paper]{amsart}
\usepackage[utf8]{inputenc}
\usepackage[T1]{fontenc}

\usepackage{amsmath,amsfonts}
\usepackage{mathtools}
\usepackage[matha,mathb]{mathabx}
\usepackage{wasysym} 
\usepackage{bbm} 
\usepackage{amsthm}
\usepackage[svgnames]{xcolor}

\usepackage[]{geometry}

\usepackage[]{hyperref}
\hypersetup{colorlinks=true,
  citecolor=DarkBlue,
  linkcolor=DarkBlue,
  linktoc=page,
  urlcolor=DarkBlue,
}

\title{Transport proofs of some functional inverse Santal\'o inequalities}
\date{\today}
\author{Matthieu Fradelizi}
\address{LAMA, Univ Gustave Eiffel, Univ Paris Est Creteil, CNRS, F-77447 Marne-la-Vall\'ee, France.}
\email{matthieu.fradelizi@univ-eiffel.fr}

\author{Nathael Gozlan \and Simon Zugmeyer}
\address{Universit\'e de Paris, CNRS, MAP5 UMR 8145, F-75006 Paris, France}
\email{nathael.gozlan@u-paris.fr, simon@zugmeyer.com}
\keywords{Blaschke-Santal\'o inequality, Mahler conjecture, Optimal Transport}
\subjclass{52A20, 52A40, 60E15}

\newtheorem{thrm}{Theorem}[section]
\newtheorem{prop}[thrm]{Proposition}
\newtheorem{coro}[thrm]{Corollary}
\newtheorem{lemm}[thrm]{Lemma}

\theoremstyle{definition}
\newtheorem{defi}[thrm]{Definition}

\theoremstyle{remark}
\newtheorem{rmrk}[thrm]{Remark}

\DeclareMathOperator{\ic}{\mathbin{\Square}} 

\newcommand{\1}{\mathbbm{1}}
\newcommand{\R}{\mathbb{R}}

\newcommand{\N}{\mathbb{N}}
\newcommand{\Sd}{\mathbb{S}}

\newcommand{\mc}{\mathcal}

\DeclareMathOperator{\supp}{supp}
\DeclareMathOperator{\dive}{div}

\DeclareMathOperator{\dom}{dom}

\DeclarePairedDelimiter{\norm}{\lVert}{\rVert} 
\DeclarePairedDelimiter{\abs}{\lvert}{\rvert}
\DeclarePairedDelimiter{\sbra}{[}{]}
\DeclarePairedDelimiter{\pare}{(}{)}

\makeatletter
\let\oldabs\abs
\def\abs{\@ifstar{\oldabs}{\oldabs*}}
\let\oldnorm\norm
\def\norm{\@ifstar{\oldnorm}{\oldnorm*}}
\let\oldsbra\sbra
\def\sbra{\@ifstar{\oldsbra}{\oldsbra*}}
\let\oldpare\pare
\def\pare{\@ifstar{\oldpare}{\oldpare*}}
\makeatother

\begin{document}
\maketitle

\begin{abstract}
In this paper, we present a simple proof of a recent result of the second author which establishes that functional inverse Santal\'o inequalities follow from Entropy-Transport inequalities. Then, using transport arguments together with elementary correlation inequalities, we prove these sharp Entropy-Transport inequalities in dimension \(1\), which therefore gives an alternative transport proof of the sharp functional Mahler conjecture in dimension $1$, for both the symmetric and the general case. We also revisit the proof of the functional inverse Santal\'o inequalities in the \(n\) dimensional unconditional case using these transport ideas.
\end{abstract}

\section{Introduction} 
The classical Blaschke-Santal\'o inequality \cite{san49} gives the following sharp relation between the volume of a convex body $K$ in $\R^n$ and the volume of its polar $K^*=\{y\in\R^n; x\cdot y\le1, \forall\ x\in K\}$: there exists $z\in\R^n$ such that $|K||(K-z)^*|\le |B_2^n|^2$, where $B_2^n$ denotes the Euclidean ball of radius one. Mahler \cite{mah39a} conjectured that the following optimal lower bound holds: 
\[
|K||K^*|\ge \frac{4^n}{n!},
\]
for any centrally symmetric convex body $K$, with equality for example if $K$ is a cube. Among general convex bodies $K$, the conjecture is that the lower bound should be reached for simplices. Both conjectures were proved by Mahler in dimension $2$ \cite{mah39b}, while the conjecture for symmetric bodies was established by Iriyeh and Shibata in dimension $3$ \cite{is20} (see also \cite{fhmrz21}). The conjectures were proved for particular families of convex bodies like unconditional convex bodies \cite{sr81, mey86}, zonoids \cite{R86,GMR88}, bodies having symmetries \cite{bf13, is21}. Bourgain and Milman \cite{bm87} (see also \cite{kup08, naz12, blo14, gpv14, ber20a, ber20b}) established an asymptotic form of the conjectures by proving that there exists a constant \(c\) such that \(|K||K^*|\ge c^n/n!\). 

Functional forms of the Mahler conjectures were proposed, where the convex bodies are replaced by log-concave functions and polar convex bodies by the Fenchel-Legendre transform.
More precisely, it is conjectured that, for any convex function \(V:\R^n\to\R\cup\{+\infty\}\) such that  \(0<\int e^{-V}\,dx<+\infty\), it holds
\[
\int e^{-V}\,dx\int e^{-V^*}\,dx\ge e^n,
\]
where the Fenchel-Legendre transform of \(V\) is defined by
\[
V^*(y)=\sup_{x\in \R^n}\left\{x\cdot y-V(x)\right\},\qquad y \in \R^n.
\]
If, in addition, $V$ is even, it is conjectured that
\[
\int e^{-V}\,dx\int e^{-V^*}\,dx\ge 4^n.
\]
These functional forms were proved in dimension \(1\) in \cite{fm08a, fm08b, fm10} and the even case was proved in dimension \(2\) in \cite{fn21}. The inequality was proved for unconditional functions in \cite{fm08a, fm08b}. 
These conjectures are slightly stronger than Mahler's conjectures for sets, because the latter are implied by the former, whereas the inequality for sets must be true in any dimension for the functional inequality to hold, as proved in~\cite{fm08a}. 

To present the class of Entropy-Transport inequalities considered in this work, we need to introduce some definitions and notations. 

The set of all Borel probability measures on $\R^n$ will be denoted by $\mathcal{P}(\R^n)$. For $k\geq 1$, we will denote by $\mathcal{P}_k(\R^n)$ the subset of $\mathcal{P}(\R^n)$ of probability measures admitting a finite moment of order $k.$ Recall that $\eta \in \mathcal{P}(\R^n)$ is said log-concave, if it admits a density with respect to the Lebesgue measure of the form $e^{-V}$, where $V:\R^n \to \R\cup\{+\infty\}$ is a lower semicontinuous convex function. The function $V$ will be referred to as the potential of $\eta$. Note that we will not consider log-concave measures supported by a strict affine subspace of $\R^n.$ The moment measure associated to a log-concave probability measure $\eta$ with potential $V$ is the measure $\nu=\nabla V\#\eta$ defined as the pushforward of $\eta$ under the map $\nabla V$: in other words, for any bounded measurable test functions, it holds
\[
\int f(x)\,\nu(dx) =\int f(\nabla V(x))\,\eta(dx).
\]
We recall that convex functions are differentiable Lebesgue almost everywhere, so that this definition makes sense.
When $\eta$ does not have full support, i.e. when $\supp(\eta)\neq \R^n$, some extra regularity will be required at the boundary. We will say that a log-concave probability measure $\eta$, with potential $V$, has an essentially continuous density, if $e^{-V}(x)=0$ for $\mc H_{n-1}$ almost all $x \in \partial\, \mathrm{Supp} (\eta)$, where $\mathrm{Supp} (\eta)$ denotes the support of $\eta.$ Note that this terminology slightly differs from the one of \cite{CEK} where it was the potential $V$ that was called essentially continuous.
\begin{defi}[Entropy-Transport inequality]\label{defi:ET} We will say that the inequality \(\mathrm{ET}_n(c)\) is satisfied for some constant \(c>0\) if, for all log-concave probability measures \(\eta_1,\eta_2\) on $\R^n$ having essentially continuous densities, it holds
\begin{equation}\label{eq:ETn}
H(\eta_1) + H(\eta_2)  \leq -n\log (ce^2) + \mc T(\nu_1,\nu_2),
\end{equation}
where \(\nu_1,\nu_2\) are the moment measures of \(\eta_1,\eta_2\).\\
Similarly, we say that \(\mathrm{ET}_{n,s}(c)\) is satisfied, if equation~\eqref{eq:ETn} holds for all log-concave measures \(\eta_1,\eta_2\) that are also symmetric (i.e. such that  \(\nu_i(A)=\nu_i(-A)\) for all measurable sets \(A\)).
\end{defi}

In the definition above, \(H(\eta)\) denotes the relative entropy of $\eta$ with respect to the Lebesgue measure (which is also equal to minus the Shannon entropy of $\eta$) and is defined by
\[
H(\eta)=\int\log\pare{\frac{d\eta}{dx}}\, d\eta.
\]
The quantity \(\mc T\) appearing in \eqref{eq:ETn} is the so-called maximal correlation optimal transport cost, defined, for any $\mu_1,\mu_2 \in \mathcal{P}_1(\R^n)$, by
\[
    \mc T(\mu_1,\mu_2) = \inf_{f\in\mc F(\R^n)} \left\{ \int f\,d\mu_1 + \int f^*\,d\mu_2 \right\},\]
where \(\mc F (\R^n)\) is the set of convex and lower semicontinuous functions $f : \R^n \to \R\cup\{+\infty\}$ which are proper (i.e. take at least one finite value). 
Since elements of $\mathcal{F}(\R^n)$ always admit affine lower bounds, note that $\int g\,d\mu_i$ makes sense in $\R\cup\{+\infty\}$ for all $g\in \mathcal{F}(\R^n)$, so that $\mc T(\mu_1,\mu_2)$ is well defined whenever $\mu_1,\mu_2 \in \mathcal{P}_1(\R^n).$ In the case where $\mu_1,\mu_2 \in \mathcal{P}_2(\R^n)$, it follows from the Kantorovich duality theorem \cite{Vil09} that 
\[
\mc T(\mu_1,\mu_2)= \sup_{X_1\sim \mu_1, X_2\sim \mu_2} \mathbb E[X_1\cdot X_2]= \sup_{\pi\in\Pi(\mu_1,\mu_2)} \int x\cdot y \,\pi(dxdy),
\]
where 
\(\Pi(\mu_1,\mu_2)\) denotes the set of probability measures on \(\R^n\times\R^n\) with marginals \(\mu_1\) and \(\mu_2\).

Definition \ref{defi:ET} is motivated by a recent result of the second author~\cite{goz21}, which states that inequality~\eqref{eq:ETn} is equivalent to the functional version of Mahler's conjecture (also called inverse Santal\'o inequality), as formulated by Klartag and Milman \cite{km05} and Fradelizi and Meyer~\cite{fm08a} that we now recall.
\begin{defi}[Inverse Santal\'o inequality]
We will say that the inequality \(\mathrm{IS}_n(c)\) is satisfied for some $c$, if for all functions \(f\in\mc F(\R^n)\) such that both \(\int e^{-f(x)}dx\) and \(\int e^{-f^*(x)}dx\) are positive, it holds 
  \begin{equation}\label{eq:ISn}
    \int e^{-f(x)}\,dx\int e^{-f^*(x)}\,dx \geq c^n.
  \end{equation}
  Similarly, we say that \(\mathrm{IS}_{n,s}(c)\) is satisfied if equation~\eqref{eq:ISn} holds for all even functions \(\mc F(\R^n)\).
\end{defi}
With this definition, the functional forms of Mahler's conjectures are \(\mathrm{IS}_{n}(e)\) and \(\mathrm{IS}_{n,s}(4)\).

\begin{thrm}[\cite{goz21}]\label{thm:goz21}
The inequality \(\mathrm{ET}_n(c)\) (resp. \(\mathrm{ET}_{n,s}(c)\)) is equivalent to \(\mathrm{IS}_{n}(c)\) (resp. \(\mathrm{IS}_{n,s}(c)\)). 
\end{thrm} 
As shown in Theorem 1.2 of \cite{goz21}, inequalities \(\mathrm{ET}_n(c)\) or \(\mathrm{ET}_{n,s}(c)\) can be restated as improved versions of the Gaussian log-Sobolev inequality. In particular, the results of \cite{fm08a, fm08b} lead to sharp lower bounds on the deficit in the Gaussian log-Sobolev inequality for  unconditional probability measures (see Theorem 1.4 of \cite{goz21}).

The main contributions of the paper are the following. In Section \ref{sec:ET-IS} we give a new proof of the implication 
\[
\mathrm{ET}_n(c) \Rightarrow \mathrm{IS}_{n}(c),
\]
and we show, in particular in Corollary \ref{cor:ETIS}, that only a restricted form of the inequality \(\mathrm{ET}_n(c)\) is enough to get \(\mathrm{IS}_{n}(c)\).
This new proof significantly simplifies the proof given in \cite{goz21}. Then, we prove in Section \ref{sec:ET1}, using transport arguments together with correlation inequalities, that \(\mathrm{ET}_{1}(e)\) and \(\mathrm{ET}_{1,s}(4)\) are satisfied. In particular, this gives new and short proofs of the sharp functional Mahler conjecture in dimension \(1\).
Finally, in Section~\ref{sec4}, we propose a short proof of \(\mathrm{IS}_{n,s}(4)\) when we restrict ourselves to unconditional functions, i.e. functions that are symmetric with respect to all coordinate hyperplanes, blending tools from this paper and the proof given in~\cite{fm08a}.

\section{Entropy-Transport and inverse Santal\'o inequalities}\label{sec:ET-IS}

\subsection{From Entropy-Transport to inverse Santal\'o inequalities}
The following result provides the key identity connecting the quantities appearing in the inverse functional inequalities to their dual transport-entropy counterparts. 
\begin{lemm}\label{lem:basic}
Let $V:\R^n\to \R\cup\{+\infty\}$ be a convex function such that \(Z:=\int e^{-V}\,dx \in (0,\infty)\) and let \(\nu\) be the moment measure of \(\eta(dx) = \frac{1}{Z}e^{-V}\,dx\). Then, it holds
\begin{equation}\label{eq:eqbasic}
- \log\pare{\int e^{-V}\,dx} = \int -V^*\,d\nu + \mc{T}(\nu,\eta)+H(\eta).
\end{equation}
\end{lemm}
\proof
According to Proposition 7 of \cite{CEK} and its proof, \(V^* \in L^1(\nu)\) and \(V\in L^1(\eta)\).
We claim that
\begin{equation}\label{eq:Tnueta}
\mc{T}(\nu,\eta) = \int V^*\,d\nu + \int V\,d\eta = \int x\cdot\nabla V(x)\,\eta(dx).
\end{equation}
Indeed, by definition of \(\mc T\), it is clear that the left hand side of \eqref{eq:Tnueta} is less than or equal to its right hand side.
On the other hand, if \(f \in \mathcal{F}(\R^n)\), then
\begin{align*}
\int f^*\,d\nu + \int f\,d\eta & = \int f^*(\nabla V(x)) + f(x)\,\eta(dx)\\
& \geq \int \nabla V(x)\cdot x\,\eta(dx)\\
& =  \int V^*(\nabla V (x))+V(x)\,\eta(dx)\\
& = \int V^*\,d\nu + \int V\,d\eta.
\end{align*}
Therefore, optimizing over \(f \in \mc F(\R^n)\), yields the converse inequality in \eqref{eq:Tnueta}. To conclude the proof of \eqref{eq:eqbasic}, just observe that
\[
H(\eta) = -\log Z - \int V\,d\eta.
\]
\endproof
It will be convenient to introduce the following class of potentials. We will denote by \(\mc V (\R^n)\) the class of all convex functions \(V:\R^n\to \R\) such that \(V^* : \R^n \to \R\) (thus \(V,V^*\) are continuous and with full domain).

\begin{rmrk}
Note that we proved Lemma~\ref{lem:basic} for convex \(V\) without assuming essential continuity. In the case where the measure is assumed to be essentially continuous, then Lemma~\ref{lem:corr} applies, and equation~\eqref{eq:eqbasic} reduces to
\[
- \log\pare{\int e^{-V}\,dx} = \int -V^*\,d\nu + n +H(\eta).
\]
This is true in particular whenever \(V\) is assumed to have full domain, i.e. to never take the value \(+\infty\). This case was already treated in the proof of Corollary 3 in \cite{goz21}, for example.
\end{rmrk}

Thanks to Lemma \ref{lem:basic}, we can show the following.
\begin{prop}\label{prop:ETIS}
Let \(V\in \mathcal{V}(\R^n)\); denote by \(\eta(dx)=\frac{1}{Z}e^{-V}\,dx\), \(\eta^*(dx)=\frac{1}{Z^*}e^{-V^*}\,dx\), where $Z,Z^*$ are the normalizing constants, and let \(\nu,\nu^*\) be the moment measures associated to \(\eta,\eta^*\). \\
If 
\begin{equation}\label{eq:ETreduced}
H(\eta)  + H(\eta^*)  \leq -n\log (ce^2) + \mc T(\nu,\nu^*),
\end{equation}
then
\[
\int e^{-V}\,dx\int e^{-V^*}\,dx \geq c^n.
\]
\end{prop}

Note that, according to e.g Lemma 4 in \cite{goz21}, if \(V\in \mc V(\R^n)\)  then \(Z:=\int e^{-V}\,dx\) and \(Z^*:=\int e^{-V^*}\,dx\) are both finite, and so the log concave probability measures \(\eta\) and \(\eta^*\) are well defined.

\proof
Applying Lemma \ref{lem:basic} and Lemma~\ref{lem:corr} to \(V\) and \(V^*\) yields
\begin{align*}
- \log\pare{\int e^{-V}\,dx} &= \int -V^*\,d\nu + \mc{T}(\nu,\eta)+H(\eta) = \int -V^*\,d\nu + n +H(\eta)\\
- \log\pare{\int e^{-V^*}\,dx}& = \int -V\,d\nu^* + n+H(\eta^*).
\end{align*}
Adding these two identities yields
\begin{align*}
- \log\pare{\int e^{-V}\,dx\int e^{-V^*}\,dx} &= -\pare{\int V^*\,d\nu +\int V\,d\nu^*}+ H(\eta)+H(\eta^*)+2n\\
& \leq -\mc T(\nu,\nu^*)+H(\eta)+H(\eta^*)+2n\\
& \leq - n\log(ce^2) + 2n = -\log(c^n),
\end{align*}
where the first inequality comes from the definition of \(\mc T(\nu,\nu^*)\) and the second inequality from~\eqref{eq:ETreduced}.

\endproof

\begin{coro}\label{cor:ETIS} Inequality \(\mathrm{IS}_n(c)\) (resp. \(\mathrm{IS}_{n,s}(c)\)) holds true as soon as for all \(V\in \mc V(\R^n)\) (resp. for all symmetric \(V\in \mc V(\R^n)\))
\[
H(\eta)  + H(\eta^*)  \leq -n\log (ce^2) + \mc T(\nu,\nu^*),
\]
where \(\eta(dx)=\frac{1}{Z}e^{-V}\,dx\), \(\eta^*(dx)=\frac{1}{Z^*}e^{-V^*}\,dx\) with $Z,Z^*$ the normalizing constants and where \(\nu,\nu^*\) are the moment measures associated to \(\eta,\eta^*\).\\
\end{coro}
\proof
According to Proposition \ref{prop:ETIS}, it holds
\[
\int e^{-V}\,dx\int e^{-V^*}\,dx \geq c^n,
\]
for all \(V\in \mc V(\R^n)\).
Let \(V\in \mc F(\R^n)\) be such that \(0<\int e^{-V}\,dx\int e^{-V^*}\,dx<\infty\).
For all \(k\geq1\), consider
\[
V_k(x) = V\ic\pare{k \frac{|\,\cdot\,|^2}{2}} (x) + \frac{|x|^2}{2k},\qquad x\in \R^n,
\]
where \(|\,\cdot\,|\) denotes the standard Euclidean norm on \(\R^n\) and \(\ic\) is the infimum convolution operator, defined by 
\[
f\ic g(x) = \inf\{f(y)+g(x-y) : y\in \R^n\},\qquad x\in \R^n.
\]
Since the infimum convolution leaves the class of convex functions stable, it is clear that \(V_k\) is still convex for all $k\geq1$. It is also clear that \(V_k\) takes finite values on \(\R^n\). 
Since \((f+g)^* = f^*\ic g^*\) and (equivalently) \((f\ic g)^* = f^*+g^*\) for all \(f,g \in \mc F(\R^n)\), it is not difficult to check that
\[
V_k^*(y) = \pare{V^*+\frac{|\,\cdot\,|^2}{2k}}\ic \pare{k \frac{|\,\cdot\,|^2}{2}}(y),\qquad y\in \R^n
\]
and so \(V_k^*\) takes finite values on \(\R^n\). In other words, \(V_k \in \mc V(\R^n)\) for all \(k\geq1\).
Since
\[
V_k \geq V\ic\pare{k \frac{|\,\cdot\,|^2}{2}} \qquad\text{and}\qquad V_k^* \geq V^*\ic \pare{k \frac{|\,\cdot\,|^2}{2}},
\]
one gets that
\[
\int e^{-V\ic\pare{k \frac{|\,\cdot\,|^2}{2}}}\,dx \int e^{-V^*\ic \pare{k \frac{|\,\cdot\,|^2}{2}}}\,dx  \geq \int e^{-V_k}\,dx \int e^{-V_k^*}\,dx \geq  c^n.
\]
Note that \(V\ic\pare{k \frac{|\,\cdot\,|^2}{2}}\) is the Moreau-Yosida approximation of \(V\). In particular, it is well known that if \(V \in \mc F (\R^n)\) then \(V\ic\pare{k \frac{|\,\cdot\,|^2}{2}}(x)\to V(x)\), for all \(x\in \R^n\), as \(k\to\infty\) (see e.g \cite[Lemma~3.6]{fn21}). Since \(V\ic\pare{k \frac{|\,\cdot\,|^2}{2}} \geq V\ic\pare{\frac{|\,\cdot\,|^2}{2}}\), it easily follows, from the dominated convergence theorem, that
\[
\int e^{-V\ic\pare{k \frac{|\,\cdot\,|^2}{2}}}\,dx \to \int e^{-V}\,dx,
\]
as \(k\to\infty\). Reasoning similarly for the other integral, one concludes that 
\[
\int e^{-V}\,dx\int e^{-V^*}\,dx\geq c^n,
\]
which completes the proof.
\endproof

\begin{rmrk}
Note that the functions \(V_k\) and \(V_k^*\) are both continuously differentiable on \(\R^n\). This follows from a well known regularizing property of the Moreau-Yosida approximation (see e.g \cite[Theorem 26.3]{Rock}). Therefore, the conclusion of Corollary \ref{cor:ETIS} is still true if the Entropy-Transport inequality \eqref{eq:ETreduced} is only assumed to hold for \(V\in\mathcal{V}_1(\R^n)\), where \(\mathcal{V}_1(\R^n)\) denotes the set of \(V\in \mathcal{V}(\R^n)\) such that \(V\) and \(V^*\) are continuously differentiable. 
\end{rmrk}

\subsection{Different equivalent formulations of inverse Santal\'o inequalities}

The following result gathers different equivalent formulations of \(\mathrm{IS}_{n}(c)\).

\begin{thrm}\label{thm:resume}
Let \(c>0\); the following statements are equivalent:
\begin{enumerate}
\item[\((i)\)] the inequality \(\mathrm{IS}_{n}(c)\) holds,
\item[\((ii)\)] the inequality \(\mathrm{ET}_n(c)\) holds,
\item[\((iii)\)] for all \(V \in \mathcal{V}(\R^n)\), 
\[
H(\eta)  + H(\eta^*)  \leq -n\log (ce^2) + \mc T(\nu,\nu^*),
\]
where \(\eta,\eta^*\) are the log-concave probability measures with respective potentials \(V,V^*\) and associated moment measures \(\nu,\nu^*\),
\item[\((iv)\)] for all \(V \in \mathcal{V}(\R^n)\), 
\[
H(\eta)  + H(\eta^*)  \leq -n\log (ce^2) + \int V^*\,d\nu + \int V\,d\nu^*,
\]
with the same notation as above.
\end{enumerate}
The same equivalence is true for \(\mathrm{IS}_{n,s}(c)\) and \(\mathrm{ES}_{n,c}(c)\) assuming in \((iii)\) and \((iv)\) that \(V \in \mc V(\R^n)\) is symmetric.
\end{thrm}
  
\proof
\((i)\Rightarrow (ii)\) follows from Theorem \ref{thm:goz21} proved in \cite{goz21}.\\
\((ii)\Rightarrow (iii)\) is straightforward.\\
\((iii)\Rightarrow (iv)\) follows from the inequality \(\mc T (\nu,\nu^*)\leq \int V^*\,d\nu + \int V\,d\nu^*\).\\
\((iv)\Rightarrow (i)\) follows from the proof of Proposition \ref{prop:ETIS} and Corollary~\ref{cor:ETIS}.
\endproof

\begin{rmrk}Let us make some comments on Theorem \ref{thm:resume}.
\begin{enumerate}
\item[(a)] The proof of \((i)\Rightarrow (ii)\) given in \cite{goz21} makes use of the following variational characterization of moment measures due to Cordero-Klartag \cite{CEK} and Santambrogio \cite{Sant16}: a measure \(\nu\) is the moment measure of a log-concave probability measure \(\eta\) with an essentially continuous density if and only if it is centered and not supported by an hyperplane; moreover, the measure \(\eta\) is the unique (up to translation) minimizer of the functional
\[
\mathcal{P}_1(\R^n)\to \R\cup\{+\infty\}:\eta \mapsto \mathcal{T}(\nu,\eta)+H(\eta).
\]
\item[(b)] In \cite{goz21}, the implication \((ii) \Rightarrow (i)\) has been established using the following duality formula: for all \(V\in \mc V(\R^n)\) such that \(\int e^{-V^*}\,dx>0\), it holds
\[
L(V):=-\log \left(\int e^{-V^*}\,dx\right) = \sup_{\nu\in \mc P_1(\R_n)} \left\{\int -V\,d\nu -K(\nu) \right\},
\]
with \(K(\nu)= \inf_{\eta \in \mc P_1(\R_n)} \{\mc T(\nu,\eta)+H(\eta)\}\), \(\nu \in \mc P_1(\R^n)\). This equality, established in \cite{goz21}, shows that the functionals \(L\) and \(K\) are in convex duality. The route followed in the present paper, based on the key Lemma \ref{lem:basic}, turns out to be simpler and more direct. 
\item[(c)] Let us finally highlight the fact that the equivalence of \((iii)\) and \((iv)\) is a bit surprising.
Namely, for a fixed \(V \in \mc F(\R^n)\), the formulation \((iii)\) is in general strictly stronger than \((iv)\), because the inequality \(\mc T (\nu,\nu^*)\leq \int V^*\,d\nu + \int V\,d\nu^*\) is strict in general. Indeed, equality here means that \((V^*,V)\) is a couple of Kantorovich potentials between \(\nu\) and \(\nu^*\). If \(\nu\) has a density with respect to Lebesgue, this means that \(\nabla V^*\) transports \(\nu\) onto \(\nu^*\) which is not true in general.  
\end{enumerate}
\end{rmrk}

\section{Proofs of Entropy-Transport inequalities in dimension $1$}\label{sec:ET1}
In this section, we show that inequalities \(\mathrm{ET}_{1,s}(4)\) and \(\mathrm{ET}_{1}(e)\) hold true. The reason why the case of dimension \(1\) is simple is that optimal transport maps for the cost \(\mathcal{T}\) are given in an explicit form.
Recall that the cumulative distribution function of  \(\mu \in \mathcal{P}(\R)\) is the function 
\[
F_\mu(x) = \mu((-\infty,x]),\qquad x\in \R.
\]
Its generalized inverse is the function denoted \(F_\mu^{-1}\) defined by
\[
F_\mu^{-1}(t) = \inf\{x : F_\mu(x) \geq t\},\qquad t \in (0,1).
\]

\begin{lemm}\label{lem:OTdim1}
Let \(\eta_1,\eta_2 \in \mathcal{P}_1(\R)\) be such that \(\mc T (\eta_1,\eta_2)\) is finite. It holds
\[
\mc T (\eta_1,\eta_2) \geq \int_0^1 F_{\eta_1}^{-1}(x)F_{\eta_2}^{-1}(x)\,dx,
\]
with equality if \(\eta_1,\eta_2 \in \mathcal{P}_2(\R)\). More generally, if \(\nu_1=S_1\#\eta_1\) and \(\nu_2=S_2\#\eta_2\) with \(S_1,S_2 : \R \to \R\) two measurable maps, and if \(\nu_1,\nu_2 \in \mathcal{P}_1(\R)\) are such that
\(\mc T (\nu_1,\nu_2)\) is finite, then
\[
\mc T (\nu_1,\nu_2) \geq \int_0^1 S_1(F_{\eta_1}^{-1})(x)S_2(F_{\eta_2}^{-1})(x)\,dx.
\]
\end{lemm}
\proof
It is well known that, if \(X\) is uniformly distributed on \((0,1)\), then \((F_{\eta_1}^{-1}(X), F_{\eta_2}^{-1}(X))\) is a coupling between \(\eta_1\) and \(\eta_2\) called the monotone coupling. Therefore, \((S_1(F_{\eta_1}^{-1}(X)), S_2(F_{\eta_2}^{-1}(X)))\) is a coupling between \(\nu_1,\nu_2\).
Suppose that \(\mc T (\nu_1,\nu_2)\) is finite, then, if \(f \in \mc F(\R)\) is such that \(f\in L^1(\nu_1)\) and \(f^*\in L^1(\nu_2)\),  Young's inequality yields
\[
f(S_1(F_{\eta_1}^{-1}(X))) + f^*(S_2(F_{\eta_2}^{-1}(X))) \geq S_1(F_{\eta_1}^{-1}(X))S_2(F_{\eta_2}^{-1}(X)).
\]
Therefore, \([S_1(F_{\eta_1}^{-1}(X))S_2(F_{\eta_2}^{-1}(X))]_+\) is integrable, and taking expectation, we get
\[
\int_0^1S_1(F_{\eta_1}^{-1}(x))S_2(F_{\eta_2}^{-1}(x))\,dx =  \mathbb{E}[S_1(F_{\eta_1}^{-1}(X))S_2(F_{\eta_2}^{-1}(X))] \leq \int f\,d\nu_1 + \int f^*\,d\nu_2.
\]
Optimizing over \(f\) gives the desired inequality.
In the case where \(S_1=S_2=\mathrm{Id}\) and \(\eta_1,\eta_2\) have finite moments of order \(2\), then it is well known that the monotone coupling is optimal for \(W_2^2\) and so also for \(\mc T\). 
\endproof

\begin{lemm}
 The inequality \(\mathrm{ET}_{1}(c)\) is satisfied as soon as for all concave functions \(f_1,f_2:[0,1]\to\R_+\) such that \(f_1(0)=f_2(0)=f_1(1)=f_2(1)=0\), 
  \begin{equation}\label{eq:OT_IS}
    \int_0^1 \log (f_1f_2)\,dx \leq -\log(e^2c) + \int_0^1 f_1'f_2'\,dx.
  \end{equation}
  Similarly, the inequality \(\mathrm{ET}_{1,s}(c)\) is satisfied as soon as inequality~\eqref{eq:OT_IS} holds for all functions \(f_1,f_2\) that are also symmetric with respect to \(1/2\), i.e. \(f_i(x)=f_i(1-x)\) for all \(x\in[0,1]\).
\end{lemm}

\begin{proof}
Let \(\eta_i(dx)=e^{-V_i}\,dx\), \(i=1,2\) be two log-concave probability measures on \(\R\) with  essentially continuous densities. This latter condition means that, for some \(-\infty \leq a_i<b_i \leq +\infty\), the convex function \(V_i\) takes finite values on \((a_i,b_i)\), is \(+\infty\) on \(\R\setminus (a_i,b_i)\) and is such that  \(V_i(x) \to +\infty\) when \(x \to a_i\) and \(x\to b_i\).
As shown in the proof of Lemma \ref{lem:basic}, 
\[
\mc T(\eta_i,\nu_i) = \int x V_i'(x)\,\eta_i(dx) =\int_{a_i}^{b_i} x V_i'(x)e^{-V_i(x)}\,dx=1,
\]
where the second equality comes from an integration by parts, thanks to the boundary conditions (see Lemma~\ref{lem:corr} in the Appendix for the case of dimension \(n\)).
To prove \(\mathrm{ET}_1(c)\), one can assume that \(\mc T(\nu_1,\nu_2)\) is finite, otherwise there is nothing to prove. Using Lemma \ref{lem:OTdim1} with \(S_i=V_i'\), we see that the inequality
\begin{equation}\label{eq:ineqinter}
H(\eta_1)+H(\eta_2)\leq -\log(ce^2) + \int_0^1 V_1'(F_{\eta_1}^{-1}(x))V_2'(F_{\eta_2}^{-1}(x))\,dx
\end{equation}
implies \(\mathrm{ET}_1(c)\).
For \(i=1,2\), define 
\[
f_i(x)= F_{\eta_i}'\circ F_{\eta_i}^{-1}(x) = \exp(-V_i\circ F_{\eta_i}^{-1}(x)),\qquad x\in (0,1).
\]
Note that, since \(F_{\eta_i}\) is strictly increasing and differentiable on \((a_i,b_i)\), the function \(F_{\eta_i}^{-1}\) is the regular inverse of the restriction of \(F_{\eta_i}\) to \((a_i,b_i)\) and is also differentiable on \((0,1)\).
Since \(F_{\eta_i}^{-1}(x) \to b_i\) as \(x\to1\) and \(\exp(-V_i(y))\to 0\) as \(y\to b_i\), one sees that \(f_i(x) \to 0\) as \(x\to 1\). Similarly, \(f_i(x) \to 0\) as \(x\to 0\). Setting \(f_i(0)=f_i(1)=0\) thus provides a continuous extension of \(f_i\) to \([0,1]\). The function \(f_i\) is moreover concave on \([0,1]\). Indeed, denoting by \(f_i'\) and \(V_i'\) the left derivatives of \(f_i,V_i\) which are well defined everywhere on \((0,1)\), we see that for all \(x\in (0,1)\),
 \[
  f_i'(x) = (F_{\eta_i}'\circ F^{-1}_{\eta_i})'(x) = \frac{F_{\eta_i}''\circ F^{-1}_{\eta_i}(x) }{F_{\eta_i}'\circ F^{-1}_{\eta_i}(x)} = - V'_i(F^{-1}_{\eta_i}(x)).
  \]
So, \(f_i'\) is decreasing on \((0,1)\), and thus \(f_i\) is concave.
Finaly, note that
\[
H(\eta_1)+H(\eta_2) = \int_0^1 \log (f_1f_2)\,dx
\]
and 
\[
\int_0^1 V_1'(F_{\eta_1}^{-1})V_2'(F_{\eta_2}^{-1})\,dx = \int_0^1 f_1'f_2'\,dx,
\]
so that inequality~\eqref{eq:ineqinter} becomes
\[
  \int_0^1 \log (f_1f_2)\,dx \leq -\log(e^2c) + \int_0^1 f_1'f_2'\,dx.
  \]
It is furthermore clear that whenever \(\eta_1,\eta_2\) are symmetric, then \(f_1,f_2\) are also symmetric with respect to \(1/2\), which concludes the proof.
\end{proof}

\begin{rmrk}
The functions \(f_i\) are related to the isoperimetric profiles of the measures \(\eta_i\) in dimension 1. Moreover, there is a one to one correspondence between log-concave measures \(\eta\) and concave \(f\) on \((0,1)\), see for example~\cite[Proposition A.1]{bob96}.  
\end{rmrk}

\subsection{The one-dimensional symmetric case}
\begin{thrm}\label{th:1ds}
 The inequality \(\mathrm{ET}_{1,s}(4)\) is satisfied and the constant \(4\) is optimal.
\end{thrm}
\begin{proof}
  Let \(f_1,f_2\) be two concave functions on \(\sbra{0,1}\), equal to zero at \(0\) and \(1\), and symmetric with respect to \(1/2\). Let us show that inequality \eqref{eq:OT_IS} holds true with \(c=4\). It is enough to prove that
  \[
  \int_0^{1/2}\log(f_1f_2)\,dx \leq -1 - \log(2) + \int_0^{1/2} f_1'f_2'\,dx.
  \]
  We use the following classical correlation inequality: if \(h,k : \R\to\R\) are two non-increasing functions (or non-decreasing), and if \(\mu\) is a finite measure on \(\R\), then
  \begin{equation}\label{eq:cor}
  \int_\R h(x)\,\mu(dx) \int_\R k(x)\,\mu(dx)\leq \mu(\R)\int_\R h(x)k(x)\,\mu(dx),
  \end{equation}
  which follows from the integration of the inequality
  \[
  (h(x)-h(y))(k(x)-k(y)) \geq 0.
  \]
  As a result, since \(f_1'\) and \(f_2'\) are non-increasing, we get, for all \(x\in[0,1]\), that
  \begin{align}\label{eq:cor2}
    f_1(x)f_2(x) = \int_0^x f_1'(t)\,dt \int_0^x f_2'(t)\,dt \leq x\int_0^x f_1'(t)f_2'(t)\,dt.
  \end{align}
  For a later use, note that this inequality holds also even if \(f_1,f_2\) are not symmetric.
  By symmetry, \(f_1'(t)f_2'(t)\geq0\) for all \(t\in\sbra{0,1/2}\), so we get
  \[
   f_1(x)f_2(x) \leq x\int_0^{1/2} f_1'(t)f_2'(t)\,dt,\qquad \forall x\in [0,1/2].
  \]
  Thus, after integrating,
  \begin{align*}
    \int_0^{1/2} \log(f_1(x)f_2(x))\,dx &\leq \int_0^{1/2} \log(x)\,dx + \frac{1}{2}\log\pare{\int_0^{1/2}f_1'(t)f_2'(t)\,dt}\\
    &\leq \frac{1}{2}\log\pare{\frac{1}{2}}-\frac{1}{2} + \int_0^{1/2}f_1'(t)f_2'(t)\,dt + \frac{1}{2}\log\pare{\frac{1}{2}}-\frac{1}{2}\\
    &= -1 - \log(2) + \int_0^{1/2}f_1'(t)f_2'(t)\,dt,
  \end{align*}
  where we used the inequality \(\log(x)-\log(1/2)\leq 2x -1\).

  To see that this inequality is sharp, we can use the functions \(f_1(x)=\min(x,1-x)\) and \(f_2\) an approximation of the constant function equal to \(1/2\). The optimal constant is reached at the limit.
\end{proof}

\begin{rmrk}
  The choice \(f_1(x)=\min(x,1-x)\) corresponds to the log-concave probability measure \(\eta(dx)=e^{-\abs x}\,dx/2\), the polar transform of which is the uniform probability measure on \([-1,1]\). These densities are the equality case in the functional Mahler inequality~\cite{fm08a}. However, the uniform probability measure on \([-1,1]\) is not an admissible measure in our case, since it is not essentially continuous, thus the optimality is only reached at the limit.
\end{rmrk}

\begin{rmrk}
  Inequality~\eqref{eq:OT_IS} is also satisfied if we assume only one of the functions to be symmetric. Indeed, if \(f_2\) is symmetric with respect to \(1/2\), define \(\tilde f_1(x)= \frac{1}{2}(f(x)+f(1-x)).\)
  On the one hand, using the concavity of the logarithm, 
  \begin{align*}
    \int_0^1 \log(\tilde f_1(x)f_2(x))\,dx &= \int_0^1 \log\tilde f_1(x)\,dx+\int_0^1 \log f_2(x)\,dx \\
    &\geq \frac{1}{2}\int_0^1 \log(f_1(x)) + \log(f_1(1-x))\,dx + \int_0^1 \log f_2(x)\,dx\\
    &= \int_0^1 \log f_1(x)\,dx + \log f_2(x)\,dx = \int_0^1 \log(f_1(x)f_2(x))\,dx,
  \end{align*}
  and on the other hand,
  \[
  \int \tilde f_1' f_2'\,dx = \frac{1}{2} \int_0^1 f_1'(x)f_2'(x)\,dx - \frac{1}{2}\int_0^1 f_1'(x)f_2'(1-x)\,dx,
  \]
  hence the claim, since \(f_2'(x)=-f_2'(1-x)\) for all \(x\in\sbra{0,1}\).
\end{rmrk}

\subsection{The one-dimensional general case}

\begin{thrm}\label{th:1d}
The inequality \(\mathrm{ET}_{1}(e)\) is satisfied and the constant \(e\) is sharp.
\end{thrm}

\begin{proof}
Let us show that, if \(f_1,f_2:[0,1]\to\R^+\) are concave functions  vanishing at \(0\) and \(1\), then
  \[
  \int_0^1 \log(f_1f_2)\,dx \leq -3 + \int_0^1 f_1'f_2'\,dx.
  \]
 Just like before, it is enough to show that
 \[
 \int_0^{1/2} \log(f_1f_2)\,dx \leq -\frac{3}{2} + \int_0^{1/2} f_1'f_2'\,dx.
 \]
 Applying the inequality \(\log (b) \leq \log (a) + \frac{(b-a)}{a}\) to \(b=f_1f_2\) and \(a=x(1-x)\), \(x\in (0,1)\), and using again the correlation inequality~\eqref{eq:cor2}, we get  
 \begin{align*}
   \int_0^{1/2} \log(f_1f_2)\,dx & \leq \int_0^{1/2} \pare{\frac{f_1(x)f_2(x)}{x(1-x)} + \log(x(1-x)) - 1}\,dx \\
   &\leq -\frac{3}{2} + \int_0^{1/2} \frac{1}{1-x}\pare{\int_0^x f_1'(t)f_2'(t)dt}\,dx \\
   &= -\frac{3}{2} + \int_0^{1/2} f_1'(t)f_2'(t)\log(2-2t)\,dt, 
 \end{align*}
 and Lemma~\ref{le:weighted_product} below concludes the proof of the inequality.

 To see that the inequality is optimal, we choose for \(f_1\) and \(f_2\) approximations of the functions \(x\mapsto x\) and \(x\mapsto 1-x\), which of course are not admissible, since they are not zero on the boundary. It is a straightforward calculation to see that equality is reached at the limit.
\end{proof}

\begin{rmrk}
  The function \(f_1(x)= x\) is the isoperimetric profile of the log-concave probability measure \(\nu(dx)=e^{-(1+x)}\1_{[-1,+\infty[}dx\), which density 
  is an equality case in the functional Mahler inequality~\cite{fm08a}.
\end{rmrk}

\begin{lemm}\label{le:weighted_product}
  Let \(f,g:\sbra{0,1}\to\R_+\) be two concave functions  vanishing at \(0\) and \(1\). The following inequality holds:
  \begin{equation}\label{eq:weighted_product}
    \int_0^{1/2} f'(t)g'(t)\log (2-2t)\,dt \leq \int_0^{1/2}f'(t)g'(t)\,dt.
  \end{equation}
\end{lemm}

\begin{proof}
For $0\le t\le 1/2$, we define $\varphi(t)=1-\log(2)-\log(1-t)$ and $\Phi(t)=\int_0^t\varphi(x)\,dx$. Notice that $\varphi$ is increasing on $[0,1/2]$ and $\varphi(0)=1-\log(2)>0$, hence $\varphi>0$ on $[0,1/2]$. Let $u=f'$ and $v=g'$. The functions $u$ and $v$ are non-increasing and satisfy $\int_0^1u\,dx=\int_0^1v\,dx=0$. Applying the correlation inequality~\eqref{eq:cor} again, and integrating with respect to the measure with density $\varphi$ on $[0,1/2]$, we get 
\[
\int_0^{1/2}\varphi\,dx\int_0^{1/2}uv\varphi\,dx\ge\int_0^{1/2}u\varphi\,dx\int_0^{1/2}v\varphi\,dx.
\]
Integrating by parts, one has 
\[
\int_0^{1/2}u\varphi\,dx=\Big[u\Phi\Big]_0^{1/2}+\int_0^{1/2}(-u')\Phi\,dx=u\left(\frac{1}{2}\right)\Phi\left(\frac{1}{2}\right)+\int_0^{1/2}(-u')\Phi\,dx.
\]
A quick calculation shows that $\Phi(1/2)=1-\log(2)=\varphi(0)$. Since $\varphi$ is increasing, it follows that $\Phi(x)\ge \varphi(0)x=\Phi(1/2)x$. Using this inequality, the fact that $u$ is non-increasing and integrating again by parts, we get 
\[
\int_0^{1/2}(-u'(x))\Phi(x)\,dx\ge \Phi\left(\frac{1}{2}\right)\int_0^{1/2}(-u'(x))x\,dx=\Phi\left(\frac{1}{2}\right)\left(-\Big[u(x)x\Big]_0^{1/2}+\int_0^{1/2}u(x)\,dx\right).
\]
Thus, using that $u$ is non-increasing again, we get
\[
\int_0^{1/2}u\varphi \,dx\ge\Phi\left(\frac{1}{2}\right)\left(\frac{1}{2}u\left(\frac{1}{2}\right)+\int_0^{1/2}u(x)\,dx\right)\ge\Phi\left(\frac{1}{2}\right)\int_0^1u(x)\,dx=0.
\]
One also has $\int_0^{1/2}v\varphi\,dx\ge0$, so we conclude that $\int_0^{1/2}uv\varphi\,dx \ge0$, which establishes \eqref{eq:weighted_product}.
\end{proof}

\section{Revisiting the unconditional case}\label{sec4}
Recall that a function \(V:\R^n \to \R\) is said unconditional if 
\[
V(x_1,\ldots,x_n) = V(|x_1|,\ldots,|x_n|),\qquad \forall x\in \R^n.
\]
The following result is due to Fradelizi and Meyer \cite{fm08a, fm08b}.
\begin{thrm}
Let \(V:\R^n \to \R\cup\{+\infty\}\) be a convex unconditional function such that \(0<\int_{\R^n} e^{-V}\,dx <\infty\) then
\begin{equation}\label{eq:uncond0}
\int_{\R^n} e^{-V}\,dx\int_{\R^n} e^{-V^*}\,dx \geq 4^n.
\end{equation}
\end{thrm}
Below, we show how Lemma \ref{lem:basic} can be used to shorten the proof of \cite{fm08a}. More precisely, from Lemma \ref{lem:basic} we quickly derive the inequality \eqref{eq:uncond1} below, which is the key step of the argument, and then the rest of the proof follows the same path as in \cite{fm08a}. 
\proof
Reasoning as in the proof of Corollary \ref{cor:ETIS}, it is enough to prove \eqref{eq:uncond0} when \(V,V^*\) have full domain and are continuously differentiable on \(\R^n\). Since \(V\) and \(V^*\) are unconditional, it is clear that \eqref{eq:uncond0} is equivalent to
\begin{equation}\label{eq:uncond}
\int_{\R_+^n} e^{-V}\,dx\int_{\R_+^n} e^{-V^*}\,dx \geq 1.
\end{equation}
Let us prove \eqref{eq:uncond} by induction on \(n\). 

- For \(n=1\), \eqref{eq:uncond} follows from Theorem \ref{th:1ds}. 

- Let \(V:\R^n\to \R\), with \(n\geq2\), satisfying the assumption of the theorem. For all \(t>0\), let \(a(t) = \int_{\R_+^n} e^{-tV}\,dx\) and \(\eta_t(dx) = \frac{1}{a(t)}e^{-tV(x)}\1_{\R_+^n}(x)\,dx\).
Applying Lemma \ref{lem:basic} to \(\eta_t\) and Jensen's inequality yields
\[
H(\eta_t) + n + \log a(t) = t\int V^*(\nabla V)\,d\eta_t \geq t V^*\left(\int_{\R_+^n} \nabla V\,d\eta_t\right),\qquad t>0.
\]
Here, we have used that \(\mathcal{T}(\nu,\eta) = \int_{\R_+^n} x\cdot \nabla V(x)e^{-V(x)}\,dx = n\) because the boundary terms in the integration by parts are \(0\).
A simple integration by parts shows that, for all \(t>0\),
\[
\int_{\R_+^n} \nabla V\,d\eta_t = \frac{G(t)}{ta(t)},
\]
where \(G(t) = (a_1(t),\ldots,a_n(t))\) and \(a_i(t) = \int_{\R_+^{n-1}} e^{-t V_i(x)}\,dx\), with 
\[
V_i(x) = V(x_1,\ldots,x_{i-1},0,x_{i+1},\ldots,x_n),\qquad x \in \R_+^{n-1}.
\]
Since \(H(\eta_t)+\log a(t) = t\frac{a'(t)}{a(t)}\), we get
\begin{equation}\label{eq:uncond1}
\frac{a'(t)}{a(t)}+\frac{n}{t} \geq V^*\left(\frac{G(t)}{ta(t)}\right),\qquad \forall t>0.
\end{equation}
Denoting \(\alpha(t) = \int_{\R_+^n} e^{-tV^*}\,dx\) and \(\Gamma(t) = (\alpha_1(t),\ldots,\alpha_n(t))\), with \(\alpha_i(t) = \int_{\R_+^{n-1}} e^{-t (V^*)_i(x)}\,dx\), a similar calculation gives
\begin{equation}\label{eq:uncond2}
\frac{\alpha'(t)}{\alpha(t)}+\frac{n}{t} \geq V\left(\frac{\Gamma(t)}{t\alpha(t)}\right),\qquad \forall t>0.
\end{equation}
Adding \eqref{eq:uncond1} and \eqref{eq:uncond2} and applying Young's inequality gives, for all \(t>0\),
\[
\frac{a'(t)}{a(t)}+\frac{\alpha'(t)}{\alpha(t)}+\frac{2n}{t}\geq V^*\left(\frac{G(t)}{ta(t)}\right) + V\left(\frac{\Gamma(t)}{t\alpha(t)}\right) \geq \frac{G(t)}{ta(t)}\cdot\frac{\Gamma(t)}{t\alpha(t)} = \frac{1}{t^2a(t)\alpha(t)} \sum_{i=1}^{n}a_i(t)\alpha_i(t).
\]
Note that for all \(1\leq i\leq n\), \((V_i)^* = (V^*)_i\) because \(V\) is non-decreasing with respect to each coordinate. By induction, for all \(1\leq i\leq n\) and \(t>0\),
\[
t^{n-1}a_i(t)\alpha_i(t)  = \int_{\R_+^{n-1}} e^{-tV_i}\,dx\int_{\R_+^{n-1}} e^{-(tV_i)^*}\,dx \geq 1.
\]
Therefore, for all \(t>0\),
\[
\frac{a'(t)}{a(t)}+\frac{\alpha'(t)}{\alpha(t)}+\frac{2n}{t} \geq \frac{n}{t^{n+1}a(t)\alpha(t)},
\]
which amounts to
\[
F'(t) \geq nt^{n-1},
\]
with \(F(t)=t^{2n}a(t)\alpha(t)\). Since \(F(0)=0\), one gets
\(F(1)\geq 1\), which is exactly \eqref{eq:uncond}.
\endproof

\section*{Appendix}

For completeness' sake, we provide here the proof of the following technical result, which mostly follows the arguments given in~\cite{CEK}.
\begin{lemm}\label{lem:corr}
    For all essentially continuous log-concave probability measure \(\eta\in\mc P(\R^n)\), its moment measure \(\nu\) satisfies \(\mc T(\eta,\nu)=n\).
\end{lemm}
\begin{proof}
    Let \(\eta(dx)=e^{-V(x)}\,dx\) be an essentially continuous probability measure, and \(\nu=\nabla V\#\eta\) its moment measure. As established in Lemma~\ref{lem:basic}, the maximal correlation is given by
    \[
    \mc T(\mu,\nu)= \int x\cdot\nabla V(x)e^{-V(x)}\,dx.
    \]
    Assuming everything is smooth, an integration by parts immediately proves that
    \[
    \mc T(\mu,\nu) = \int \dive(x)e^{-V(x)}\,dx - \int_{\partial\dom V} x\cdot n_{\dom V}(x)e^{-V(x)}\,d\mc H_{n-1}(x) = n, 
    \]
    since \(e^{-V(x)}=0\) for \(\mc H_{n-1}\)-almost all \(x\in\dom V\). In the general case, however, \(V\) is only Lipschitz on the interior of its domain. Thus, let us choose \(x_0\) in the interior of the domain of \(V\). According to~\cite[Lemma 4]{CEK}, 
    \[
    \int \nabla V(x)e^{-V(x)}\,dx = 0
    \]
    by essential continuity, and thus
    \[
    \mc T(\mu,\nu)= \int x\cdot\nabla V(x)e^{-V(x)}\,dx = \int (x-x_0)\cdot\nabla V(x)e^{-V(x)}\,dx.
    \]
    Convexity of \(V\) implies that the function \(x\mapsto (x-x_0)\cdot \nabla V(x)\) is bounded from below by some constant (which is, of course, integrable against \(\eta\)), and so, if \((K_N)_{N\in\N}\) is an increasing sequence of compact sets such that \(\bigcup_N K_N=\dom V\),
    \[
    \int (x-x_0)\cdot\nabla V(x)e^{-V(x)}\,dx = \lim_{N\to\infty} \int_{K_N} (x-x_0)\cdot\nabla V(x)e^{-V(x)}\,dx.
    \]
    For \(N\in\N\), with $N>\min V$, the sets \(\{ V\leq N\}\) are convex, closed because of lower semicontinuity, with non empty interior since $\int e^{-V}>0$, bounded since \(\lim_{\abs x\to+\infty }V(x)=+\infty\) and strictly increasing by the essential continuity of \(e^{-V}\). Since convex bodies may be approximated by smooth convex bodies (see \cite[Lemma 2.3.2]{hor07}), we can find a sequence \((K_N)\) of smooth convex bodies such that 
    \[
    \{ V \leq N \} \subset K_N \subset \{ V \leq 2N \} 
    \]
    for all $N>\min V$. It is clear that then \(\bigcup_N K_N = \dom V\). Since \(K_N\) is smooth, and \(V\) is Lipschitz on \(K_N\), the divergence theorem applies:
    \[
    \int_{K_N} (x-x_0)\cdot\nabla V(x)e^{-V(x)}\,dx = \int_{K_N} \dive(x)e^{-V(x)}\,dx - \int_{\partial K_N} n_{K_N}(x)\cdot(x-x_0)e^{-V(x)}\,d\mc H_{n-1}(x), 
    \]
    where \(n_{K_N}(x)\) is the outer normal vector to \(K_N\) at \(x\). Clearly,
    \[
    \lim_{N\to\infty} \int_{K_N} \dive(x)e^{-V(x)}\,dx = n\lim_{N\to+\infty}\eta(K_N)=n,
    \]
    and we will show that the second term converges towards zero.
    To that end, note that since \(e^{-V(x)}\) is integrable, there exist constants \(a>0\) and \(b\) such that \(V(x) \geq a\abs x + b\). As an immediate consequence, for all \(N>b\), the sublevel set \(\{V\leq N\}\) is included in the ball of center \(0\) and of radius \(R_N=(N-b)/a\). Hence, whenever \(N\) is large enough so that \(x_0\in K_N\),
    \begin{align*}
    \abs{\int_{\partial K_N} n_{K_N}(x)\cdot(x-x_0)e^{-V(x)}\,d\mc H_{n-1}(x)} 
    &\leq\int_{\partial K_N} \abs{x-x_0}e^{-V(x)}\,d\mc H_{n-1}(x)\\
    &\leq  2R_{2N}\,e^{-N} \mc H_{n-1}(\partial K_N).
    \end{align*}
    Finally, if \(K,L\) are two convex bodies such that \(K\subset L\), then \(\mc H_{n-1}(\partial K)\leq \mc H_{n-1}(\partial L)\) (see \cite[(5.25)]{sch14}), and so \(\mc H_{n-1}(\partial K_N)\leq R_{2N}^{n-1}\mc H_{n-1}(\Sd^{n-1})\), which is enough to conclude that
    \[
    \abs{\int_{\partial K_N} n_{K_N}(x)\cdot(x-x_0)e^{-V(x)}\,d\mc H_{n-1}(x)} 
    \leq  p(N)e^{-N}, 
    \]
    where \(p\) is some polynomial, which proves our claim.
\end{proof}

\section*{Acknowledgments} The authors were supported by a grant of the Simone and Cino Del Duca foundation. The third author has benefited from a post doctoral position funded by the Simone and Cino Del Duca foundation. This research has been conducted within the FP2M federation (CNRS FR 2036). The authors warmly thank Anne Estrade, the head of MAP 5, for making the meetings that led to this publication possible in the midst of the epidemic situation.

\bibliographystyle{alpha}
{\footnotesize
\bibliography{notes}}

\newcommand{\etalchar}[1]{$^{#1}$}
\begin{thebibliography}{FHM{\etalchar{+}}21}

\bibitem[Ber20a]{ber20a}
B.~Berndtsson.
\newblock Bergman kernels for {P}aley-{W}iener spaces and {N}azarov's proof of
  the {B}ourgain-{M}ilman theorem, 2020.
\newblock arXiv:2008.00837.

\bibitem[Ber20b]{ber20b}
B.~Berndtsson.
\newblock Complex integrals and {K}uperberg's proof of the {B}ourgain-{M}ilman
  theorem, 2020.
\newblock arXiv:2008.00838.

\bibitem[BF13]{bf13}
F.~Barthe and M.~Fradelizi.
\newblock The volume product of convex bodies with many hyperplane symmetries.
\newblock {\em Amer. J. Math.}, 135(2):311--347, 2013.

\bibitem[Blo14]{blo14}
Z.~Blocki.
\newblock A lower bound for the {B}ergman kernel and the {B}ourgain-{M}ilman
  inequality.
\newblock In {\em Geometric aspects of functional analysis}, volume 2116 of
  {\em Lecture Notes in Math.}, pages 53--63. Springer, Cham, 2014.

\bibitem[BM87]{bm87}
J.~Bourgain and V.~D. Milman.
\newblock New volume ratio properties for convex symmetric bodies in {${\bf
  R}^n$}.
\newblock {\em Invent. Math.}, 88(2):319--340, 1987.

\bibitem[Bob96]{bob96}
S.~Bobkov.
\newblock Extremal properties of half-spaces for log-concave distributions.
\newblock {\em Ann. Probab.}, 24(1):35--48, 1996.

\bibitem[CEK15]{CEK}
D.~Cordero-Erausquin and B.~Klartag.
\newblock Moment measures.
\newblock {\em J. Funct. Anal.}, 268(12):3834--3866, 2015.

\bibitem[FHM{\etalchar{+}}21]{fhmrz21}
M.~Fradelizi, A.~Hubard, M.~Meyer, E.~Roldán-Pensado, and A.~Zvavitch.
\newblock Equipartitions and {M}ahler volumes of symmetric convex bodies, 2021.
\newblock arXiv:1904.10765, to appear in Amer. J. Math.

\bibitem[FM08a]{fm08b}
M.~Fradelizi and M.~Meyer.
\newblock Increasing functions and inverse {S}antal\'{o} inequality for
  unconditional functions.
\newblock {\em Positivity}, 12(3):407--420, 2008.

\bibitem[FM08b]{fm08a}
M.~Fradelizi and M.~Meyer.
\newblock Some functional inverse {S}antal\'{o} inequalities.
\newblock {\em Adv. Math.}, 218(5):1430--1452, 2008.

\bibitem[FM10]{fm10}
M.~Fradelizi and M.~Meyer.
\newblock Functional inequalities related to {M}ahler's conjecture.
\newblock {\em Monatsh. Math.}, 159(1-2):13--25, 2010.

\bibitem[FN21]{fn21}
M.~Fradelizi and E.~Nakhle.
\newblock The functional form of {M}ahler conjecture for even log-concave
  functions in dimension $2$, 2021.
\newblock arXiv:2101.08065.

\bibitem[GMR88]{GMR88}
Y.~Gordon, M.~Meyer, and S.~Reisner.
\newblock Zonoids with minimal volume-product---a new proof.
\newblock {\em Proc. Amer. Math. Soc.}, 104(1):273--276, 1988.

\bibitem[Goz21]{goz21}
N.~Gozlan.
\newblock The deficit in the {G}aussian {L}og-{S}obolev inequality and inverse
  {S}antal\'o inequalities.
\newblock {\em International Mathematics Research Notices}, 2021.
\newblock To appear.

\bibitem[GPV14]{gpv14}
A.~Giannopoulos, G.~Paouris, and B.~H. Vritsiou.
\newblock The isotropic position and the reverse {S}antal\'{o} inequality.
\newblock {\em Israel J. Math.}, 203(1):1--22, 2014.

\bibitem[H\"07]{hor07}
L.~H\"{o}rmander.
\newblock {\em Notions of convexity}.
\newblock Modern Birkh\"{a}user Classics. Birkh\"{a}user Boston, Inc., Boston,
  MA, 2007.
\newblock Reprint of the 1994 edition.

\bibitem[IS20a]{is20}
H.~Iriyeh and M.~Shibata.
\newblock Symmetric {M}ahler's conjecture for the volume product in the
  {$3$}-dimensional case.
\newblock {\em Duke Math. J.}, 169(6):1077--1134, 2020.

\bibitem[IS20b]{is21}
Hiroshi Iriyeh and Masataka Shibata.
\newblock Minimal volume product of three dimensional convex bodies with
  various discrete symmetries, 2020.

\bibitem[KM05]{km05}
B.~Klartag and V.~D. Milman.
\newblock Geometry of log-concave functions and measures.
\newblock {\em Geom. Dedicata}, 112:169--182, 2005.

\bibitem[Kup08]{kup08}
G.~Kuperberg.
\newblock From the {M}ahler conjecture to {G}auss linking integrals.
\newblock {\em Geom. Funct. Anal.}, 18(3):870--892, 2008.

\bibitem[Mah39a]{mah39b}
K.~Mahler.
\newblock Ein {M}inimalproblem f\"{u}r konvexe {P}olygone.
\newblock {\em Mathematica (Zutphen)}, 1939.

\bibitem[Mah39b]{mah39a}
K.~Mahler.
\newblock Ein \"{U}bertragungsprinzip f\"{u}r konvexe {K}\"{o}rper.
\newblock {\em \v{C}asopis P\v{e}st. Mat. Fys.}, 68:93--102, 1939.

\bibitem[Mey86]{mey86}
M.~Meyer.
\newblock Une caract\'{e}risation volumique de certains espaces norm\'{e}s de
  dimension finie.
\newblock {\em Israel J. Math.}, 55(3):317--326, 1986.

\bibitem[Naz12]{naz12}
F.~Nazarov.
\newblock The {H}\"{o}rmander proof of the {B}ourgain-{M}ilman theorem.
\newblock In {\em Geometric aspects of functional analysis}, volume 2050 of
  {\em Lecture Notes in Math.}, pages 335--343. Springer, Heidelberg, 2012.

\bibitem[Rei86]{R86}
S.~Reisner.
\newblock Zonoids with minimal volume-product.
\newblock {\em Math. Z.}, 192(3):339--346, 1986.

\bibitem[{Roc}97]{Rock}
R.~Tyrrell {Rockafellar}.
\newblock {\em {Convex analysis}}.
\newblock Princeton, NJ: Princeton University Press, 1997.

\bibitem[San49]{san49}
L.~A. Santal\'{o}.
\newblock An affine invariant for convex bodies of {$n$}-dimensional space.
\newblock {\em Portugal. Math.}, 8:155--161, 1949.

\bibitem[{San}16]{Sant16}
F.~{Santambrogio}.
\newblock {Dealing with moment measures via entropy and optimal transport}.
\newblock {\em {J. Funct. Anal.}}, 271(2):418--436, 2016.

\bibitem[Sch14]{sch14}
Rolf Schneider.
\newblock {\em Convex bodies: the {B}runn-{M}inkowski theory}, volume 151 of
  {\em Encyclopedia of Mathematics and its Applications}.
\newblock Cambridge University Press, Cambridge, expanded edition, 2014.

\bibitem[SR81]{sr81}
J.~Saint-Raymond.
\newblock Sur le volume des corps convexes sym\'{e}triques.
\newblock In {\em Initiation {S}eminar on {A}nalysis: {G}. {C}hoquet-{M}.
  {R}ogalski-{J}. {S}aint-{R}aymond, 20th {Y}ear: 1980/1981}, volume~46 of {\em
  Publ. Math. Univ. Pierre et Marie Curie}, pages Exp. No. 11, 25. Univ. Paris
  VI, Paris, 1981.

\bibitem[{Vil}09]{Vil09}
C.~{Villani}.
\newblock {\em {Optimal transport. Old and new}}, volume 338.
\newblock Berlin: Springer, 2009.

\end{thebibliography}

\end{document}